\documentclass[a4paper,12pt]{article}
\usepackage{authblk}
\usepackage{amsmath}
\usepackage{amsfonts}
\usepackage{amssymb}
\usepackage{amsthm}
\usepackage{graphicx}
\newtheorem{theorem}{Theorem}

\newtheorem{remark}[theorem]{Remark}

\theoremstyle{remark}

\newcommand{\id}{ {1\!\!\!\:1 } }

\begin{document}

\title{Existence theorems in the geometrically  non-linear 6-parametric theory of elastic plates}

\author{Mircea B\^{\i}rsan%
\thanks{\,Mircea B\^{\i}rsan,  Fakult\"at f\"ur  Mathematik, Universit\"at Duisburg-Essen, Campus Essen, Universit\"atsstr. 2, 45141 Essen, Germany, email: mircea.birsan@uni-due.de ; and
Department of Mathematics, University ``A.I. Cuza'' of Ia\c{s}i, 700506 Ia\c{s}i,  Romania}\,\,\,
and\, Patrizio Neff\,%
\thanks{Corresponding author: Patrizio Neff, Lehrstuhl f\"ur Nichtlineare Analysis und Modellierung, Fakult\"at f\"ur  Mathematik, Universit\"at Duisburg-Essen, Campus Essen, Universit\"atsstr. 2, 45141 Essen, Germany, email: patrizio.neff@uni-due.de, Tel.: +49-201-183-4243}
}


\maketitle

\begin{abstract}
In this paper we show the existence of global minimizers for the geometrically exact, non-linear equations of elastic plates, in the framework of the general 6-parametric shell theory. A characteristic feature of this model for shells is the appearance of two independent kinematic fields: the translation vector field and the rotation tensor field (representing in total 6 independent scalar kinematic variables). For isotropic plates, we prove the existence theorem by applying the direct methods of the calculus of variations. Then, we generalize our existence result to the case of anisotropic plates. We also present a detailed comparison with a previously established Cosserat plate model.
\end{abstract}

\smallskip

\noindent \textbf{Mathematics Subject Classification  (MSC 2010):} 74K20, 74K25, 74G65, 74G25.

\section{Introduction}\label{sect1}

In this paper we present an existence theorem for the geometrically non-linear equations of elastic plates, in the framework of the  6-parametric shell theory.

The general (6-parametric) non-linear theory of shells, originally proposed by Reissner \cite{Reissner74}, has been considerably developed in the last 30 years. This theory and the most results in the field have been presented in the books of Libai and Simmonds \cite{Libai98} and recently  Chr\'o\'scielewski, Makowski and  Pietraszkiewicz \cite{Pietraszkiewicz-book04}.
The model is based on a dimension-reduction procedure of the three-dimensional formulation of the problem to the two-dimensional one, and is expressed through stress resultants and work-averaged deformation fields defined on the shell base surface. Thus, the local equilibrium equations for shells are derived by an exact through-the-thickness integration of the three-dimensional independent  balance laws for linear momentum and angular momentum. The deformation of the shell is characterized by two independent kinematic fields: the translation (displacement) vector   and the rotation tensor. The appearance of the rotation tensor as an independent kinematic field variable is one of the most characteristic features of this general shell theory. In this respect, we mention that the kinematical structure of the non-linear  6-parametric shell theory is identical to that of the Cosserat shell model (i.e., the material surface with a triad of rigidly rotating directors attached to any point) proposed initially by the Cosserat brothers \cite{Cosserat09neu} and developed subsequently by Zhilin \cite{Zhilin76}, Zubov \cite{Zubov97}, Altenbach and Zhilin \cite{Altenbach04}, Eremeyev and Zubov \cite{Eremeyev08}, B\^{\i}rsan and Altenbach \cite{Birsan10}, among others. A related Cosserat shell-model has been establish recently by Neff \cite{Neff_plate04_cmt,Neff_plate07_m3as} using the so-called derivation approach.

The subject of derivation and justification of the non-linear, geometrically exact equations for plates and shells has been treated in many works, and the
existence of solutions has been investigated
using a variety of methods, such as the method of formal asymptotic expansions or the $\Gamma$-convergence analysis, see e.g. \cite{Simo89.1,Simo92,Sansour92,Aganovic06,Aganovic07,Tiba02,Paroni06,Paroni06b}.
For an extensive treatment of this topic, as well as many bibliographic references, we refer to the books of Ciarlet \cite{Ciarlet97,Ciarlet00}.
To the best of our knowledge, one cannot find  in the literature   any  existence theorem  for the non-linear 6-parametric  theory of plates or shells developed in \cite{Libai98,Pietraszkiewicz-book04}.  In our work, we describe the non-linear equations of elastic plates as a two-fields minimization problem of the total potential energy and we prove the existence of minimizers by applying the direct methods of the calculus of variations. The first result concerning the existence of minimizers for a geometrically exact (6-parametric) Cosserat plate model has been presented by the second author in \cite{Neff_plate04_cmt,Neff_plate07_m3as}. Due to differences in notation, this result has not been much noticed. In this paper, we employ similar techniques as in \cite{Neff_plate04_cmt} and adapt the existence proof to the general 6-parametric plate equations.

In the framework of the linearized 6-parametric theory, the existence of weak solutions for  micropolar elastic shells has been proved recently in \cite{EremeyevLebedev11}. We mention that the kinematic structure of the general 6-parametric shell theory differs from that of the so-called Cosserat surfaces, i.e. material surfaces with one or more deformable directors \cite{Green65,Naghdi72,Antman95}. In particular, the kinematics of Cosserat surfaces with one deformable director \cite{Naghdi72,Rubin00} leaves indefinite the drilling rotation about the director, while the general 6-parametric shell theory is able to take into  account such drilling rotation.
For the linear theory of Cosserat surfaces, the existence theorems have been established in \cite{Davini75,Birsan08,BirsanZAMM08}.

Here is the outline of our paper: In Section \ref{sect2} we briefly review the field equations of the non-linear 6-parametric plate theory. Then, in Section \ref{sect3} we prove the existence theorem for isotropic plates. The generalization of the existence result for anisotropic plates is presented in Section \ref{sect4}. We also show that the existence theorem remains valid in the case of some alternative relaxed boundary conditions for the rotation field. In Section \ref{sect5} we present a detailed comparison between the non-linear 6-parametric plate and the Cosserat plate model proposed and investigated by the second author in \cite{Neff_plate04_cmt,Neff_plate07_m3as}. Although this Cosserat model for plates has been obtained independently by a   formal dimensional reduction of a finite-strain three-dimensional micropolar model (see also \cite{Neff_Forest07}), we show here that the kinematical variables and the strain measures of the two models essentially coincide. Moreover, the expressions of the elastic strain energies  become identical, provided we make a suitable identification of the constitutive coefficients for isotropic plates in the two approaches.

The linearized version of this model has also been investigated in \cite{Neff_plate04_cmt,Neff_plate07_m3as,Neff_Danzig05,Neff_Chelminski_ifb07,Neff_Hong_Reissner08} and its relations to the Reissner--Mindlin, Kirchhoff--Love and other classical plate models have been discussed.

\section{Basic equations of geometrically exact  elastic plates}\label{sect2}

The governing equations   of the general 6-parametric non-linear theory of shells have been derived  in \cite{Libai98,Pietraszkiewicz-book04,Antman95} by direct integration of the two independent  fundamental principles of continuum mechanics: the three-dimensional balance laws of linear momentum and angular momentum. In this section we summarize these equations, specialized here for the case of plates. To this aim, we employ mainly the notations introduced in \cite{Libai98,Pietraszkiewicz04,Pietraszkiewicz-book04}.

Consider an elastic plate, which is a three-dimensional body  identified in the reference (undeformed) configuration with a region $\Omega=\{(x_1,x_2,z)\,|\,(x_1,x_2)\in \omega,\,z\in \big[-\frac{h}{2}\,,\frac{h}{2}\,\big]\}$ of the Euclidean space. Here $h>0$ is the thickness of the plate and $\omega\subset \mathbb{R}^2$ is a bounded, open domain with Lipschitz boundary $\partial\omega$. Relative to an inertial frame $(O,\boldsymbol{e}_i)$, with $\boldsymbol{e}_i$ orthonormal vectors ($i=1,2,3$), the position vector $\boldsymbol{p}$ of any point of $\Omega$ can be written as
\begin{equation}\label{1}
    \boldsymbol{p}(x_\alpha,z)=\boldsymbol{x}+z\,\boldsymbol{e}_3\,,\quad \boldsymbol{x}=x_\alpha\boldsymbol{e}_\alpha\,,\quad (x_1,x_2)\in\omega,\quad z\in\big[-\frac{h}{2}\,,\frac{h}{2}\,\big].
\end{equation}
Throughout this paper, we employ the usual convention of summation over repeated indices. The Latin indices $i,j,...$ take the values $\{1,2,3\}$ and the Greek indices $\alpha,\beta,...$ range over the set $\{1,2\}$.

\begin{figure}
\begin{center}
\includegraphics[scale=1]{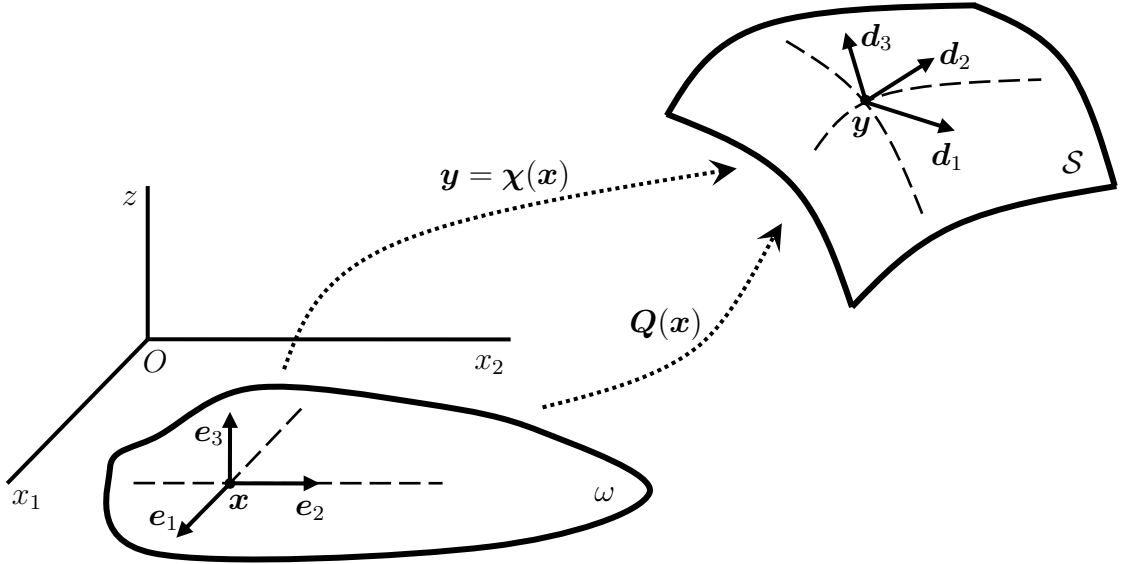}
\put(-369,77){$O$} \put(-418,27){$x_1$} \put(-245,78){$x_2$} \put(-377,140){$z$}
\put(-337,25){$\boldsymbol{x} $} \put(-367,19){$\boldsymbol{e}_1$} \put(-312,23){$\boldsymbol{e}_2$} \put(-350,51){$\boldsymbol{e}_3$}
\put(-104,168){$\boldsymbol{y} $} \put(-74,155){$\boldsymbol{d}_1$} \put(-71,193){$\boldsymbol{d}_2$} \put(-100.5,200){$\boldsymbol{d}_3$}
\put(-187,92){$\boldsymbol{Q}( \boldsymbol{x})$} \put(-258,148){$\boldsymbol{y}=\boldsymbol{\chi}(\boldsymbol{x})$}
\put(-200,28){$\omega$} \put(-25,152){$\mathcal{S}$}
\caption{The reference base surface $\,\omega\,$ of the plate and the deformed surface $\mathcal{S}$, described by the surface deformation mapping $\boldsymbol{y}=\boldsymbol{\chi}(\boldsymbol{x})$ and the independent rotation tensor field $\boldsymbol{Q}(\boldsymbol{x})$.}
\label{Fig1}       
\end{center}
\end{figure}

In the deformed configuration, the base surface of the plate (shell) is represented by the position vector $\boldsymbol{y}=\boldsymbol{\chi}(\boldsymbol{x})$, where $\,\boldsymbol{\chi}:\omega\subset \mathbb{R}^2\rightarrow\mathbb{R}^3$ is the surface deformation mapping. Let the vector field $\boldsymbol{u}(\boldsymbol{x})$ represent the  translations (displacements) and the proper orthogonal tensor field $\boldsymbol{Q}(\boldsymbol{x})$ designate the rotations of the shell cross-sections. Then the deformed configuration of the plate is given by
\begin{equation}\label{2}
    \boldsymbol{y}=\boldsymbol{\chi}(\boldsymbol{x})=\boldsymbol{x}+\boldsymbol{u}(\boldsymbol{x}),\qquad \boldsymbol{d}_i=\boldsymbol{Q}\,\boldsymbol{e}_i\,,\quad i=1,2,3.
\end{equation}
The vectors $\boldsymbol{d}_i$ introduced in \eqref{2} are three orthonormal directors (see Figure 1) attached to any point of the deformed base surface $\mathcal{S}=\boldsymbol{\chi}(\omega)$. Thus, the rotation field $\boldsymbol{Q}(\boldsymbol{x})\in SO(3)$ can be written as
\begin{equation}\label{3}
   \boldsymbol{Q}=\boldsymbol{d}_i\otimes\boldsymbol{e}_i\,.
\end{equation}

According to the Lagrangian description, let $\boldsymbol{f}$ and $\boldsymbol{c}$ be the external surface resultant force and couple vectors applied at any point $\boldsymbol{y}\in\mathcal{S}$, but measured per unit area of $\omega$. Also, let $\boldsymbol{n}_\nu=\boldsymbol{N}\boldsymbol{\nu}$ and $\boldsymbol{m}_\nu=\boldsymbol{M}\boldsymbol{\nu}$ be the internal contact stress and couple resultant vectors defined at an arbitrary boundary curve $\partial G\subset \mathcal{S}$, but measured per unit length of the undeformed boundary $\partial \gamma\subset\omega$. (We have denoted here by $G=\boldsymbol{\chi}(\gamma)$ and $\boldsymbol{\nu}$ is the external unit normal vector to $\partial \gamma$ lying in the plane of $\omega$.) Here, the tensors $\boldsymbol{N}=N_{i\alpha}\boldsymbol{e}_i\otimes\boldsymbol{e}_\alpha$ and $\boldsymbol{M}=M_{i\alpha}\boldsymbol{e}_i\otimes\boldsymbol{e}_\alpha$ are the internal surface stress resultant and stress couple resultant tensors (of the first Piola--Kirchhoff stress tensor type).
Then, the local equilibrium equations are given in the form \cite{Pietraszkiewicz-book02,Pietraszkiewicz-book04}
\begin{equation}\label{4}
    \mathrm{Div}_s\, \boldsymbol{N}+\boldsymbol{f}=\boldsymbol{0},\qquad \mathrm{Div}_s\, \boldsymbol{M} + \mathrm{axl}(\boldsymbol{N}\boldsymbol{F}^T-\boldsymbol{F}\boldsymbol{N}^T)
    +\boldsymbol{c}=\boldsymbol{0},
\end{equation}
where $\boldsymbol{F}= \mathrm{Grad}_s\boldsymbol{y}= \boldsymbol{y}_{,\alpha}\otimes\boldsymbol{e}_\alpha$ is the surface gradient of deformation and $\mathrm{Div}_s \boldsymbol{N}=N_{i\alpha,\alpha} \boldsymbol{e}_i$ ,
$\mathrm{Div}_s \boldsymbol{M}=M_{i\alpha,\alpha} \boldsymbol{e}_i\,$. As usual, we denote the partial derivative with respect to $x_\alpha$ by $(\cdot)_{,\alpha}=\frac{\partial}{\partial x_\alpha}\,(\cdot)$.
In \eqref{4} the superscript $(\,\cdot)^T$ denotes the transpose and $\mathrm{axl}(\boldsymbol{A})$ is the axial vector of any three-dimensional skew-symmetric tensor $\boldsymbol{A}$, given by
\begin{equation}\label{5}
    \mathrm{axl}(\boldsymbol{A})= A_{32}\boldsymbol{e}_1+A_{13}\boldsymbol{e}_2+ A_{21}\boldsymbol{e}_3\,,\quad\mathrm{for}\quad \boldsymbol{A}=A_{ij}\boldsymbol{e}_i\otimes\boldsymbol{e}_j\,,\quad \boldsymbol{A}^T=-\boldsymbol{A},
\end{equation}
such that $\,\,\boldsymbol{A}\,\boldsymbol{v}=\mathrm{axl}(\boldsymbol{A})\times \boldsymbol{v}$, for any vector $\boldsymbol{v}$. The corresponding weak form of the local balance equations has been presented in \cite{Libai98} Chap. VIII, or in \cite{Pietraszkiewicz04}.

To formulate the boundary conditions, we take a disjoint partition of the boundary curve $\partial \omega=\partial \omega_d\cup \partial \omega_f$ , $\partial \omega_d\cap \partial \omega_f=\emptyset$, with length$(\partial \omega_d)>0$. We consider the following boundary conditions \cite{Pietraszkiewicz-book04,Pietraszkiewicz11}
\begin{equation}\label{6}
    \boldsymbol{u}-\boldsymbol{u}^*=\boldsymbol{0},\qquad \boldsymbol{Q}-\boldsymbol{Q}^*=\boldsymbol{0}\quad\mathrm{along}\,\,\,\partial\omega_d\,,
\end{equation}
\begin{equation}\label{7}
    \boldsymbol{N}\boldsymbol{\nu}-\boldsymbol{n}^*=\boldsymbol{0},\qquad \boldsymbol{M}\boldsymbol{\nu}-\boldsymbol{m}^*=\boldsymbol{0}\quad\mathrm{along}\,\,\,\partial\omega_f\,,
\end{equation}
where $\boldsymbol{n}^*$ and $\boldsymbol{m}^*$ are the external boundary resultant force and couple vectors applied along the part $\partial\omega_f$ of the boundary $\partial\omega$.
In general, from a modelling point of view, it is a difficult task to specify boundary conditions for the rotation $\boldsymbol{Q}\,$ in \eqref{6} at the   boundary $\,\partial\omega_d\,$. In \cite{Pietraszkiewicz04} the following requirement is considered: in relations \eqref{6} the functions $\boldsymbol{u}^*$ and $\boldsymbol{Q}^*$ defined on $\partial\omega_d$ should be found from the Dirichlet boundary conditions $\boldsymbol{u}_{3D}(x_\alpha,z)=\boldsymbol{u}_{3D}^*(x_\alpha,z)$ for the three-dimensional body $\Omega$, at any point $(x_\alpha,z)\in \partial \Omega_d = \partial \omega_d\times[-\frac{h}{2},\frac{h}{2}]$. Thus, the functions $\boldsymbol{u}^*$ and $\boldsymbol{Q}^*$ should be determined from the condition that the work done along
$\partial\omega_d$ by the resultant stress and couple vectors $\boldsymbol{n}_{\nu\,}$, $\boldsymbol{m}_\nu$ on the translation $\boldsymbol{u}^*$ and rotation $\boldsymbol{Q}^*$ be the same as the work done along $\partial \Omega_d$ by the nominal three-dimensional stress vector $\boldsymbol{t}_\nu(x_\alpha,z)$ on the translation $\boldsymbol{u}_{3D}^*(x_\alpha,z)$.

In the general resultant theory of shells,  the   strain measures are the strain tensor $\boldsymbol{E}$ and the bending tensor $\boldsymbol{K}$, given by \cite{Pietraszkiewicz-book04,Pietraszkiewicz04,Eremeyev11}
\begin{equation}\label{8}
    \boldsymbol{E}= \boldsymbol{Q}^T(\boldsymbol{\varepsilon}_\alpha\otimes\boldsymbol{e}_\alpha),\qquad \boldsymbol{\varepsilon}_\alpha=\boldsymbol{y}_{,\alpha}- \boldsymbol{d}_\alpha\,,
\end{equation}
\begin{equation}\label{9}
    \boldsymbol{K}= \boldsymbol{Q}^T(\boldsymbol{\varkappa}_\alpha\otimes\boldsymbol{e}_\alpha),\qquad \boldsymbol{\varkappa}_\alpha=\mathrm{axl}(\boldsymbol{Q}_{,\alpha}\boldsymbol{Q}^T).
\end{equation}
We mention that the kinematical structure \eqref{8},\eqref{9} of the general shell theory is identical with that of the classical version of the Cosserat shell \cite{Cosserat09neu,Zhilin76,Altenbach04,Eremeyev08,Birsan10}.
In the case of plates, these strain tensors can be written in component form relative to the tensor basis $\{\boldsymbol{e}_i\otimes\boldsymbol{e}_\alpha\}$ as
\begin{equation}\label{10}
      {}\hspace{-112pt}\boldsymbol{E}=E_{i\alpha}\boldsymbol{e}_i\otimes\boldsymbol{e}_\alpha= (\boldsymbol{y}_{,\alpha}\cdot\boldsymbol{d}_i-\delta_{i\alpha}) \boldsymbol{e}_i\otimes\boldsymbol{e}_\alpha\,,
\end{equation}
\begin{equation}\label{11}
\begin{array}{lr}
     \boldsymbol{K}=K_{i\alpha}\boldsymbol{e}_i\otimes\boldsymbol{e}_\alpha=\frac{1}{2}\,e_{ijk}(\boldsymbol{d}_{j,\alpha}\cdot\boldsymbol{d}_k) \boldsymbol{e}_i\otimes\boldsymbol{e}_\alpha \\ \quad
    \,\, = (\boldsymbol{d}_{2,\alpha}\cdot\boldsymbol{d}_3) \boldsymbol{e}_1\otimes\boldsymbol{e}_\alpha + (\boldsymbol{d}_{3,\alpha}\cdot\boldsymbol{d}_1) \boldsymbol{e}_2\otimes\boldsymbol{e}_\alpha + (\boldsymbol{d}_{1,\alpha}\cdot\boldsymbol{d}_2) \boldsymbol{e}_3\otimes\boldsymbol{e}_\alpha\,,
     \end{array}
\end{equation}
where $\delta_{i\alpha}$ is the Kronecker symbol and $e_{ijk}$ is the permutation symbol.

According to the hyperelasticity assumption,  the constitutive equations for elastic plates are given in the form \cite{Libai98,Pietraszkiewicz-book04,Eremeyev06}
\begin{equation}\label{12}
    \boldsymbol{N}=\boldsymbol{Q}\,\dfrac{\partial\, W}{\partial \boldsymbol{E}}\,\,,\qquad \boldsymbol{M}=\boldsymbol{Q}\,\dfrac{\partial\, W}{\partial \boldsymbol{K}}\,\,,
\end{equation}
where
\begin{equation}\label{13}
    W=W(\boldsymbol{E},\boldsymbol{K})
\end{equation}
is the strain energy density.
Then, the equations \eqref{4} are the Euler-Lagrange equations corresponding to the minimization problem of the total potential energy.
In order to characterize the material of the elastic plate, one has to specify the expression of the potential energy function \eqref{13}.
In the paper \cite{Eremeyev06} the conditions for invariance of the strain energy  under change of the reference placement are discussed and the local symmetry group is established. The structure of the local symmetry group puts some constraints on the form of $W$, which allows one to simplify the expression of $W$. From this representation, the strain energy density corresponding to physically linear isotropic plates is given by (see \cite{Eremeyev06}, Sect. 10)
\begin{equation}\label{14}
\begin{array}{c}
    W(\boldsymbol{E},\boldsymbol{K})= W_{\mathrm{mb}}(\boldsymbol{E})+ W_{\mathrm{bend}}(\boldsymbol{K}),\\
    2 W_{\mathrm{mb}}(\boldsymbol{E})=  \alpha_1\mathrm{tr}^2 \boldsymbol{E}_{\parallel} +\alpha_2  \mathrm{tr} \boldsymbol{E}^2_{\parallel}    + \alpha_3 \mathrm{tr}(\boldsymbol{E}_{\parallel}^T  \boldsymbol{E}_{\parallel} )  + \alpha_4     \boldsymbol{n} \boldsymbol{E} \boldsymbol{E}^T  \boldsymbol{n}, \\
    2 W_{\mathrm{bend}}(\boldsymbol{K})= \beta_1\mathrm{tr}^2 \boldsymbol{K}_{\parallel} +\beta_2  \mathrm{tr} \boldsymbol{K}^2_{\parallel}    + \beta_3 \mathrm{tr}(\boldsymbol{K}_{\parallel}^T  \boldsymbol{K}_{\parallel} )  + \beta_4     \boldsymbol{n} \boldsymbol{K} \boldsymbol{K}^T  \boldsymbol{n}, \\
    \boldsymbol{E}_{\parallel}= \boldsymbol{E}- (\boldsymbol{n}\otimes\boldsymbol{n}) \boldsymbol{E},\qquad \boldsymbol{K}_{\parallel}= \boldsymbol{K}- (\boldsymbol{n}\otimes\boldsymbol{n}) \boldsymbol{K},\qquad \boldsymbol{n}=\boldsymbol{e}_3\,,
\end{array}
\end{equation}
where the coefficients $\alpha_k$ , $\beta_k$ ($k=1,2,3,4$) are constant material parameters.

\begin{remark}\label{rem1}
In the works \cite{Pietraszkiewicz-book04,Pietraszkiewicz10} the authors have employed a particular form of the expression \eqref{14} for the strain energy density in the isotropic homogeneous case, namely
\begin{equation}\label{15}
\begin{array}{l}
    2W(\boldsymbol{E},\boldsymbol{K})= \,\,\,C\big[\,\nu \,\mathrm{tr}^2 \boldsymbol{E}_{\parallel} +(1-\nu)\, \mathrm{tr}(\boldsymbol{E}_{\parallel}^T  \boldsymbol{E}_{\parallel} )\big]  + \alpha_{s\,}C(1-\nu) \, \boldsymbol{n} \boldsymbol{E} \boldsymbol{E}^T  \boldsymbol{n} \\
    \qquad\qquad\qquad\,\, +\,D \big[\,\nu\,\mathrm{tr}^2 \boldsymbol{K}_{\parallel} + (1-\nu)\, \mathrm{tr}(\boldsymbol{K}_{\parallel}^T  \boldsymbol{K}_{\parallel} )\big]  + \alpha_{t\,}D(1-\nu) \,    \boldsymbol{n} \boldsymbol{K} \boldsymbol{K}^T  \boldsymbol{n},
\end{array}
\end{equation}
where the coefficients are given by
\begin{equation}\label{16}
    C=\dfrac{E\,h}{1-\nu^2}\,,\qquad D=\dfrac{E\,h^3}{12(1-\nu^2)}\,,\qquad \alpha_s=\dfrac{5}{6}\,,\qquad \alpha_t=\dfrac{7}{10}\,.
\end{equation}
Here $E$ is the Young modulus, $\nu$ is the Poisson ratio, $C$ is the stretching (membrane) stiffness of the plate, and $D$ is the bending stiffness. The values of the two shear correction factors $\alpha_s$ and $\alpha_t$ from \eqref{16} have been determined in \cite{Pietraszkiewicz10} through a numerical treatment of several non-linear shell structures. We observe that the form \eqref{15} of the strain energy density $W$ can be obtained from the more general representation \eqref{14} by choosing  the following coefficients
\begin{equation}\label{17}
\begin{array}{l}
    \alpha_1=C\nu=\,\dfrac{2\lambda\mu}{\lambda\!+\!2\mu}\,h,\quad \alpha_2=0,\quad \alpha_3=C(1\!-\!\nu)=2\mu h, \quad \alpha_4=\alpha_sC(1\!-\!\nu)=2\mu\alpha_s h,\\
    \beta_1=D\nu=\,\dfrac{\lambda\mu}{\lambda\!+\!2\mu}\,\dfrac{h^3}{6},\quad\! \beta_2=0,\quad\!\! \beta_3=D(1\!-\!\nu)=\dfrac{\mu h^3}{6}\,, \quad \! \beta_4=\alpha_tD(1\!-\!\nu)=\dfrac{\mu\alpha_t h^3}{6},\\
\end{array}
\end{equation}
where $\mu$  and $\lambda$  are the elastic Lam\'e moduli of the isotropic and homogeneous material.
\end{remark}

\section{Existence theorem for isotropic plates}\label{sect3}

We employ the usual notations for the Lebesgue space $L^2(\omega)$ and the Sobolev space $H^1(\omega)$, endowed with their usual norms $\|\cdot\|_{L^2(\omega)}$ and $\|\cdot\|_{H^1(\omega)}\,$. We denote the set of proper orthogonal tensors by $SO(3)$ and designate the set of (three-dimensional) translation vectors by $\mathbb{R}^3$ and the set of second-order tensors by $\mathbb{R}^{3\times3}$. The functional spaces of vectorial or tensorial  functions will be denoted by $\boldsymbol{L}^2(\omega, \mathbb{R}^3)$, $\boldsymbol{H}^1(\omega, \mathbb{R}^3)$, and respectively $\boldsymbol{L}^2(\omega, \mathbb{R}^{3\times3})$, $\boldsymbol{H}^1(\omega, \mathbb{R}^{3\times3})$. For tensorial functions with range in $SO(3)$, we employ the notations  $\boldsymbol{L}^2(\omega, SO(3))$ and  $\boldsymbol{H}^1(\omega, SO(3))$.
We also use  the classical notations for the norms $\|\boldsymbol{v}\|=(\boldsymbol{v}\cdot\boldsymbol{v})^{1/2}, \forall \,\boldsymbol{v}\in \mathbb{R}^3$, and $\|\boldsymbol{X}\|^2=\mathrm{tr}(\boldsymbol{X}\boldsymbol{X}^T), \forall \,\boldsymbol{X}\in \mathbb{R}^{3\times 3}$.

Let us define the admissible set $\mathcal{A}$ by
\begin{equation}\label{18}
    \mathcal{A}=\big\{(\boldsymbol{y},\boldsymbol{Q})\in\boldsymbol{H}^1(\omega, \mathbb{R}^3)\times\boldsymbol{H}^1(\omega, SO(3))\,\,\big|\,\,\,  \boldsymbol{y}_{\big| \partial\omega_d}=\boldsymbol{y}^*, \,\,\boldsymbol{Q}_{\big| \partial\omega_d}=\boldsymbol{Q}^* \big\}.
\end{equation}
The boundary conditions in \eqref{18} are to be understood in the sense of traces.
We assume the existence of a function $\Lambda(\boldsymbol{u},\boldsymbol{Q})$ representing the potential of the external surface loads $\boldsymbol{f}, \boldsymbol{c}$, and boundary loads $\boldsymbol{n}^*, \boldsymbol{m}^*$ \cite{Pietraszkiewicz04}.

Consider the   two-field minimization problem associated to the deformation of elastic plates: find the pair $(\hat{\boldsymbol{y}},\hat{\boldsymbol{Q}})\in \mathcal{A}$ which realizes the minimum of the functional
\begin{equation}\label{19}
    I(\boldsymbol{y},\boldsymbol{Q})=\int_\omega W(\boldsymbol{E},\boldsymbol{K})\,\mathrm{d}\omega - \Lambda(\boldsymbol{u},\boldsymbol{Q})\qquad\mathrm{for}\qquad (\boldsymbol{y},\boldsymbol{Q})\in \mathcal{A}.
\end{equation}
Here the strain tensor $\boldsymbol{E}$ and the bending tensor $\boldsymbol{K}$ are expressed in terms of $(\boldsymbol{y},\boldsymbol{Q})$ by relations \eqref{8} and \eqref{9}.
The variational principle of total potential energy relative to the functional \eqref{19} has been presented in \cite{Pietraszkiewicz04}, Sect.2.

The external loading potential $\Lambda(\boldsymbol{u},\boldsymbol{Q})$  is decomposed additively
\begin{equation}\label{20}
    \Lambda(\boldsymbol{u},\boldsymbol{Q})=\Lambda_\omega(\boldsymbol{u},\boldsymbol{Q}) + \Lambda_{\partial\omega_f}(\boldsymbol{u},\boldsymbol{Q}),
\end{equation}
where $\Lambda_\omega(\boldsymbol{u},\boldsymbol{Q})$ is the potential of the external surface loads $\boldsymbol{f}, \boldsymbol{c}$, while $\Lambda_{\partial\omega_f}(\boldsymbol{u},\boldsymbol{Q})$ is the potential of the external boundary loads $\boldsymbol{n}^*, \boldsymbol{m}^*$, which are taken in the form
\begin{equation}\label{21}
\begin{array}{l}
    \Lambda_\omega(\boldsymbol{u},\boldsymbol{Q})= \displaystyle{\int_\omega} \boldsymbol{f}\cdot \boldsymbol{u}\, \mathrm{d}\omega + \Pi_\omega(\boldsymbol{Q}), \\
    \Lambda_{\partial\omega_f}(\boldsymbol{u},\boldsymbol{Q})= \displaystyle{\int_{\partial\omega_f}} \boldsymbol{n}^*\cdot \boldsymbol{u}\, \mathrm{d}s + \Pi_{\partial\omega_f}(\boldsymbol{Q}).
    \end{array}
\end{equation}
The load potential functions $\,\,\Pi_\omega:\boldsymbol{L}^2(\omega,SO(3))\rightarrow\mathbb{R}$ and $\,\,\Pi_{\partial\omega_f}:\boldsymbol{L}^2( \omega ,SO(3))\rightarrow\mathbb{R}$  are assumed to be continuous and bounded operators, whose expressions are not given explicitly.

We are now able to present the main existence result concerning the deformation of isotropic elastic plates.
\begin{theorem}\label{th1}
Assume that the external loads and the boundary data satisfy the regularity conditions
\begin{equation}\label{22}
    \boldsymbol{f}\in\boldsymbol{L}^2(\omega,\mathbb{R}^3),\quad  \boldsymbol{n}^*\in \boldsymbol{L}^2(\partial\omega_f,\mathbb{R}^3), \quad \boldsymbol{y}^*\in\boldsymbol{H}^1(\omega ,\mathbb{R}^3),\quad \boldsymbol{Q}^*\in\boldsymbol{H}^1(\omega, SO(3)).
\end{equation}
Consider the minimization problem \eqref{18}, \eqref{19} for isotropic plates, i.e. when the strain energy density
$W$ is given by the relations \eqref{14}. If the constitutive coefficients satisfy the conditions
\begin{equation}\label{23}
\begin{array}{l}
    2\alpha_1+\alpha_2+\alpha_3>0,\quad \alpha_2+\alpha_3>0,\quad \alpha_3-\alpha_2>0, \quad\alpha_4>0,\\
    2\beta_1+\beta_2+\beta_3>0,\quad \beta_2+\beta_3>0,\quad \beta_3-\beta_2>0,\quad  \beta_4>0,
\end{array}
\end{equation}
then the problem \eqref{18}, \eqref{19} admits at least one minimizing solution pair $(\hat{\boldsymbol{y}},\hat{\boldsymbol{Q}})\in \mathcal{A}$.
\end{theorem}
\begin{proof} To prove this assertion, we apply the direct methods of the calculus of variations. First, we observe that for any $\boldsymbol{Q}\in SO(3)$ we have $\|\boldsymbol{Q}\|^2=3$ and, hence, $\|\boldsymbol{Q}\|_{L^2(\omega)}$  is bounded independent of $\boldsymbol{Q}$. In view of the conditions \eqref{22}$_{1,2}$ and the boundedness of $\,\Pi_\omega $ and $\,\Pi_{\partial\omega_f}\,$, we derive from \eqref{20}, \eqref{21} that there exist some positive constants $C_i>0$ such that
\begin{equation*}
\begin{array}{l}
    |\Lambda(\boldsymbol{u},\boldsymbol{Q})|\leq \|\boldsymbol{f}\|_{L^2(\omega)}\|\boldsymbol{u}\|_{L^2(\omega)} +\|\boldsymbol{n}^*\|_{L^2(\partial\omega_f)} \|\boldsymbol{u}\|_{L^2(\partial\omega_f)} +| \Pi_\omega(\boldsymbol{Q})|+ | \Pi_{\partial\omega_f}(\boldsymbol{Q})| \\
     \qquad \qquad\leq C_1 \|\boldsymbol{u}\|_{L^2(\omega)}+C_2\|\boldsymbol{u}\|_{H^1(\omega)} +C_3+C_4\,,
\end{array}
\end{equation*}
which means that there exists a constant $C>0$ with
\begin{equation}\label{24}
     |\Lambda(\boldsymbol{u},\boldsymbol{Q})|\leq \,\,C\,\big(\,\|\boldsymbol{y}\|_{H^1(\omega)}+1\big),\quad\forall\,(\boldsymbol{y},\boldsymbol{Q})\in\mathcal{A}.
\end{equation}

From \eqref{14}$_2$ we observe that $W_{\mathrm{mb}}(\boldsymbol{E})$ is a quadratic form in the strain variables $E_{i\alpha}$ ($i=1,2,3;\, \alpha=1,2$) given by \eqref{10}. More precisely, we may write
\begin{equation}\label{26}
\begin{array}{l}
    W_{\mathrm{mb}}(\boldsymbol{E})= \tilde{W}_{\mathrm{str}}(E_{i\alpha})=\frac{1}{2}(\alpha_1+\alpha_2+\alpha_3)(E_{11}^2+E_{22}^2)+\frac{1}{2}\,\alpha_3(E_{12}^2+E_{21}^2)\\
    \qquad\qquad\qquad\qquad\qquad+\frac{1}{2}\,\alpha_4 (E_{31}^2+E_{32}^2)+\alpha_1E_{11}E_{22}+\alpha_2E_{12}E_{21}\,.
\end{array}
\end{equation}
The quadratic form $\tilde{W}_{\mathrm{str}}(E_{i\alpha})$ given by \eqref{26} is positive definite if and only if the conditions \eqref{23}$_{1-4}$ on the coefficients $\alpha_k$ are satisfied. Then, by virtue of the relations \eqref{23}$_{1-4}$ we infer that there exists a constant $c_1>0$ such that
$$\tilde{W}_{\mathrm{str}}(E_{i\alpha})\geq c_1\sum_{i=1}^3 \sum_{\alpha=1}^2 E_{i\alpha}^2\,\,,\qquad \forall\,E_{i\alpha}\in\mathbb{R},$$
or, equivalently,
\begin{equation}\label{27}
    W_{\mathrm{mb}}(\boldsymbol{E})\geq  c_1\|\boldsymbol{E}\|^2, \qquad\forall\,\boldsymbol{E}=E_{i\alpha}\boldsymbol{e}_i\otimes\boldsymbol{e}_\alpha\,,\quad E_{i\alpha}\in\mathbb{R}.
\end{equation}
Analogously, from the conditions \eqref{23}$_{5-8}$ on the coefficients $\beta_k$ we deduce that there exists a constant $\bar{c}_1>0$ such that
\begin{equation}\label{28}
    W_{\mathrm{bend}}(\boldsymbol{K})\geq \bar{c}_1 \|\boldsymbol{K}\|^2, \qquad\forall\,\boldsymbol{K}=K_{i\alpha}\boldsymbol{e}_i\otimes\boldsymbol{e}_\alpha\,,\quad K_{i\alpha}\in\mathbb{R}.
\end{equation}
On the other hand, in view of \eqref{8} we observe that
$$\|\boldsymbol{E}\|^2= \mathrm{tr}(\boldsymbol{E}\boldsymbol{E}^T)=  \mathrm{tr}\big[(\boldsymbol{\varepsilon}_\alpha\cdot\boldsymbol{\varepsilon}_\beta) \boldsymbol{e}_\alpha\otimes \boldsymbol{e}_\beta\big]= \boldsymbol{\varepsilon}_\alpha\cdot\boldsymbol{\varepsilon}_\alpha = \boldsymbol{y}_{,\alpha}\cdot\boldsymbol{y}_{,\alpha}- 2 \boldsymbol{y}_{,\alpha}\cdot\boldsymbol{d}_{\alpha} +2,
$$
since $\|\boldsymbol{d}_\alpha\|=1$. Then, the Cauchy-Schwarz inequality yields
\begin{equation*}
    \displaystyle{\int_\omega} \| \boldsymbol{E}\|^2\, \mathrm{d}\omega \geq \|\boldsymbol{y}_{,1}\|^2_{L^2(\omega)} +\|\boldsymbol{y}_{,2}\|^2_{L^2(\omega)} -2\sqrt{2a}\,\big( \, \|\boldsymbol{y}_{,1}\|^2_{L^2(\omega)} +\|\boldsymbol{y}_{,2}\|^2_{L^2(\omega)}\big)^{1/2}\! +2a, 
\end{equation*}
or
\begin{equation}\label{29}
    \|\boldsymbol{E}\|_{L^2(\omega)}^2 \geq \,\|\boldsymbol{F}\|_{L^2(\omega)}^2 -2\sqrt{2a}\,\,\|\boldsymbol{F}\|_{L^2(\omega)} +2a,
\end{equation}
where $a=\mathrm{area}(\omega)$ and $\boldsymbol{F}=\mathrm{Grad}_s\boldsymbol{y}=\boldsymbol{y}_{,\alpha}\otimes\boldsymbol{e}_{\alpha}$ is the surface gradient of deformation introduced previously.

We show now that the functional $I(\boldsymbol{y},\boldsymbol{Q})$ is bounded from below over $\mathcal{A}$. Indeed, for any $(\boldsymbol{y},\boldsymbol{Q})\in\mathcal{A}$ we  use consecutively the inequalities \eqref{24}, \eqref{27}, \eqref{29} to write
\begin{equation*}
\begin{array}{l}
    I(\boldsymbol{y},\boldsymbol{Q})\geq  \displaystyle{\int_\omega} W_{\mathrm{mb}}(\boldsymbol{E})\,\mathrm{d}\omega - \Lambda(\boldsymbol{u},\boldsymbol{Q}) \geq \displaystyle{\int_\omega} c_1\|\boldsymbol{E}\|^2\,\mathrm{d}\omega - C\big(\,\|\boldsymbol{y}\|_{H^1(\omega)}+1\big) \\
    \qquad\qquad \geq c_1 \big( \|\boldsymbol{F}\|_{L^2(\omega)}^2 -2\sqrt{2a}\,\,\|\boldsymbol{F}\|_{L^2(\omega)} +2a\big)  - C\big(\,\|\boldsymbol{y}\|_{H^1(\omega)}+1\big)
\end{array}
\end{equation*}
so that
\begin{equation}\label{30}
    I(\boldsymbol{y},\boldsymbol{Q})\geq\, k_1\, \|\boldsymbol{F}\|_{L^2(\omega)}^2 -k_2\,\|\boldsymbol{y}\|_{H^1(\omega)} - k_3\,,\qquad \forall\,(\boldsymbol{y},\boldsymbol{Q})\in\mathcal{A},
\end{equation}
for some constants $k_1,k_2,k_3$ with $k_1>0$, $k_2>0$. Using the Poincar\'e inequality for the field $\boldsymbol{y}-\boldsymbol{y}^*\in \boldsymbol{H}^1(\omega, \mathbb{R}^3)$ (with $\boldsymbol{y}-\boldsymbol{y}^*=\boldsymbol{0}$ on $\partial\omega_f$) in the form
$$ \|\,\mathrm{Grad}_s(\boldsymbol{y}-\boldsymbol{y}^*)\|^2_{L^2(\omega)}=  \|\, (\boldsymbol{y}-\boldsymbol{y}^*)_{,1}\|^2_{L^2(\omega)} + \|\, (\boldsymbol{y}-\boldsymbol{y}^*)_{,2}\|^2_{L^2(\omega)} \geq c_P^+ \|\, \boldsymbol{y}-\boldsymbol{y}^*\|^2_{H^1(\omega)}\,\,,
$$
then from \eqref{30} we deduce that there exist some constants $K_1>0, K_2>0$ and $K_3$ such that
\begin{equation}\label{31}
    I(\boldsymbol{y},\boldsymbol{Q})\geq\, K_1\, \|\boldsymbol{y}-\boldsymbol{y}^*\|_{H^1(\omega)}^2 -K_2\,\|\boldsymbol{y}-\boldsymbol{y}^*\|_{H^1(\omega)} - K_3\,,\qquad \forall\,(\boldsymbol{y},\boldsymbol{Q})\in\mathcal{A}.
\end{equation}
Relation \eqref{31} shows that $I$ is bounded from below over $\mathcal{A}$, and thus there exists an infimizing sequence $\big\{\boldsymbol{y}^k,\boldsymbol{Q}^k\big\}_{k=1}^{\,\infty}\subset\mathcal{A}$ with
\begin{equation}\label{32}
    \lim_{k\rightarrow\infty}I(\boldsymbol{y}^k,\boldsymbol{Q}^k)= \inf\{ \, I(\boldsymbol{y},\boldsymbol{Q})\,|\,\,(\boldsymbol{y},\boldsymbol{Q})\in\mathcal{A}\}.
\end{equation}
According to the hypotheses \eqref{22}$_{3,4}$ we have $I(\boldsymbol{y}^*,\boldsymbol{Q}^*)<\infty$. In view of \eqref{31}, we may write for $I(\boldsymbol{y}^k,\boldsymbol{Q}^k)$
\begin{equation}\label{33}
 \infty>   I(\boldsymbol{y}^*,\boldsymbol{Q}^*) \geq I(\boldsymbol{y}^k,\boldsymbol{Q}^k)\geq K_1\, \|\boldsymbol{y}^k-\boldsymbol{y}^*\|_{H^1(\omega)}^2 -K_2\,\|\boldsymbol{y}^k-\boldsymbol{y}^*\|_{H^1(\omega)} - K_3\,,\quad \forall\,k\in\mathbb{N},
\end{equation}
and hence, the sequence $\big\{ \boldsymbol{y}^k\big\}_{k=1}^{\infty}$ is bounded in $\boldsymbol{H}^1(\omega, \mathbb{R}^3)$.
Consequently, we can extract a subsequence of $\big\{ \boldsymbol{y}^k\big\}_{k=1}^{\infty}$, not relabeled, which converges weakly to an element $\hat{ \boldsymbol{y}}$ in $\boldsymbol{H}^1(\omega , \mathbb{R}^3)$, i.e.
\begin{equation}\label{34}
    \boldsymbol{y}^k  \rightharpoonup \hat{ \boldsymbol{y}} \quad\mathrm{in}\quad \boldsymbol{H}^1(\omega, \mathbb{R}^3),\quad \mathrm{for}\,\,\, k\rightarrow\infty,
\end{equation}
and moreover it converges strongly in $\boldsymbol{L}^2(\omega, \mathbb{R}^3)$ by Rellich's selection principle
\begin{equation}\label{35}
     \boldsymbol{y}^k \rightarrow\hat{ \boldsymbol{y}} \quad\mathrm{in}\quad \boldsymbol{L}^2(\omega, \mathbb{R}^3), \quad \mathrm{with}\,\,\, \hat{ \boldsymbol{y}}\in \boldsymbol{H}^1(\omega , \mathbb{R}^3).
\end{equation}
On the other hand, from \eqref{33} and \eqref{19}, \eqref{24} it follows that the sequence $\int_\omega W_{\mathrm{mb}}(\boldsymbol{E}^k) \mathrm{d}\omega$ is bounded independent of $k\in\mathbb{N}$. Then, from \eqref{27} we deduce that $\big\{ \boldsymbol{E}^k\big\}_{k=1}^{\infty}$ is a bounded sequence in $\boldsymbol{L}^2(\omega, \mathbb{R}^{3\times3})$. Therefore, there exists a subsequence (not relabeled) and an element $\hat{ \boldsymbol{E}}\in \boldsymbol{L}^2(\omega, \mathbb{R}^{3\times3})$ such that
\begin{equation}\label{36}
    \boldsymbol{E}^k  \rightharpoonup  \hat{\boldsymbol{E}} \quad\mathrm{in}\quad \boldsymbol{L}^2(\omega, \mathbb{R}^{3\times3}), \quad \mathrm{with}\quad \boldsymbol{E}^k= \boldsymbol{Q}^{k,T} \big[(\boldsymbol{y}^k_{,\alpha}-\boldsymbol{Q}^k\boldsymbol{e}_\alpha)\otimes \boldsymbol{e}_\alpha\big].
\end{equation}
Similarly, from \eqref{33} we deduce that $\int_\omega W_{\mathrm{bend}}(\boldsymbol{K}^k) \mathrm{d}\omega$ is a bounded sequence and, in view of \eqref{28}, the sequence $\big\{ \boldsymbol{K}^k\big\}_{k=1}^{\infty}$ is bounded in $\boldsymbol{L}^2(\omega, \mathbb{R}^{3\times3})$. Taking into account \eqref{9} and \eqref{5}, we observe that $\|\boldsymbol{K}\|^2=\frac{1}{2}\,(\|\boldsymbol{Q}_{,1}\|^2+\|\boldsymbol{Q}_{,2}\|^2)$ and $\|\boldsymbol{Q}\|^2=3$. Then, it follows that the sequences $\big\{\boldsymbol{Q}_{,\alpha}^k\big\}_{k=1}^{\infty}$ are bounded in $\boldsymbol{L}^2(\omega, \mathbb{R}^{3\times3})$ and, hence,
 $\big\{\boldsymbol{Q}^k\big\}_{k=1}^{\infty}$  is a bounded sequence in $\boldsymbol{H}^1(\omega, \mathbb{R}^{3\times3})$. Consequently, there exists a subsequence (not relabeled) and an element $\hat{\boldsymbol{Q}}\in\boldsymbol{H}^1(\omega, \mathbb{R}^{3\times3})$ such that
\begin{equation}\label{37}
    \boldsymbol{Q}^k \rightharpoonup       \hat{\boldsymbol{Q}}    \quad\mathrm{in}\quad \boldsymbol{H}^1(\omega, \mathbb{R}^{3\times3}) , \quad\mathrm{and}\quad          \boldsymbol{Q}^k  \rightarrow    \hat{\boldsymbol{Q}} \quad\mathrm{in}\quad \boldsymbol{L}^2(\omega, \mathbb{R}^{3\times3}).
\end{equation}

Moreover, we observe that  $\boldsymbol{L}^2(\omega,SO(3))$ is a closed subset of $\boldsymbol{L}^2(\omega, \mathbb{R}^{3\times3})$. Indeed, for any sequence $\big\{\boldsymbol{R}^k\big\}_{k=1}^{\infty}\subset \boldsymbol{L}^2(\omega,SO(3))$ which converges to an element  $\,\boldsymbol{R}\,$  in $\boldsymbol{L}^2(\omega, \mathbb{R}^{3\times3})$, one can show that $\,\, \boldsymbol{R}^k \boldsymbol{R}^T \rightarrow\id\,\,$ in $\boldsymbol{L}^2(\omega, \mathbb{R}^{3\times3})$, and also $\,\, \boldsymbol{R}^k \boldsymbol{R}^T \rightarrow\boldsymbol{R} \boldsymbol{R}^T\,\,$ in $\boldsymbol{L}^1(\omega, \mathbb{R}^{3\times3})$. It follows that $\,
\boldsymbol{R} \boldsymbol{R}^T=\id\,$ holds, which means that $\,\boldsymbol{R}\in \boldsymbol{L}^2(\omega,SO(3))$.

Consequently, from \eqref{37}$_2$ and $\,\boldsymbol{Q}^k\in \!\boldsymbol{L}^2(\omega,SO(3))$ we obtain that $\,\hat{\boldsymbol{Q}}\in \!\boldsymbol{L}^2(\omega,SO(3))$.

By virtue of the boundedness of $\big\{ \boldsymbol{K}^k\big\}_{k=1}^{\infty}$ in $\boldsymbol{L}^2(\omega, \mathbb{R}^{3\times3})$, there exists a subsequence (not relabeled) and an element $\hat{\boldsymbol{K}}\in\boldsymbol{L}^2(\omega, \mathbb{R}^{3\times3})$ such that
\begin{equation}\label{38}
\boldsymbol{K}^k \rightharpoonup \hat{\boldsymbol{K}}\quad\mathrm{in}\quad \boldsymbol{L}^2(\omega, \mathbb{R}^{3\times3}), \quad \mathrm{with}\quad  \boldsymbol{K}^k= \boldsymbol{Q}^{k,T} \big[\mathrm{axl}(\boldsymbol{Q}^k_{,\alpha}  \boldsymbol{Q}^{k,T} )\otimes \boldsymbol{e}_\alpha\big].
\end{equation}
Concerning the (weak) limits $\,\hat{\boldsymbol{y}}$, $\hat{\boldsymbol{Q}}$, $\hat{\boldsymbol{E}}\,$ and $\,\hat{\boldsymbol{K}}\,$ specified by relations \eqref{34}--\eqref{38}, it remains to show that they satisfy the equations
\begin{equation}\label{40}
     \hat{   \boldsymbol{E}}=
     \hat{\boldsymbol{Q}}^{T} \big[(\hat{\boldsymbol{y}}_{,\alpha}\!- \hat{\boldsymbol{Q}}\boldsymbol{e}_\alpha)\otimes \boldsymbol{e}_\alpha\big],\qquad
   \hat{ \boldsymbol{K}}= \hat{\boldsymbol{Q}}^T\big[\,\mathrm{axl}(\hat{ \boldsymbol{Q}}_{,\alpha}\hat{\boldsymbol{Q}}^T)\otimes\boldsymbol{e}_\alpha\big].
\end{equation}
Indeed, for any test function $\boldsymbol{\varphi}\in \boldsymbol{C}_0^\infty(\omega, \mathbb{R}^{3})$, we may write
\begin{equation*}
\begin{array}{l}
    \displaystyle{\int_\omega}\big(\boldsymbol{Q}^{k,T} \boldsymbol{y}_{,\alpha}^k - \hat{\boldsymbol{Q}}^T \hat{\boldsymbol{y}}_{,\alpha}\big)\cdot\boldsymbol{\varphi} \,\mathrm{d}\omega=
    \displaystyle{\int_\omega}\big[ \big( \boldsymbol{Q}^{k,T}- \hat{\boldsymbol{Q}}^T\big) \boldsymbol{y}_{,\alpha}^k +  \hat{\boldsymbol{Q}}^T \big( \boldsymbol{y}_{,\alpha}^k - \hat{\boldsymbol{y}}_{,\alpha}\big)\big]\cdot\boldsymbol{\varphi} \,\mathrm{d}\omega \vspace{2pt}\\
    =\displaystyle{\int_\omega}\big[ \,\big\langle \boldsymbol{Q}^{k}- \hat{\boldsymbol{Q}}\,,\,  \boldsymbol{y}_{,\alpha}^k \otimes\boldsymbol{\varphi}\rangle +  \big( \boldsymbol{y}_{,\alpha}^k - \hat{\boldsymbol{y}}_{,\alpha}\big) \cdot\hat{\boldsymbol{Q}}\boldsymbol{\varphi}\big]\mathrm{d}\omega  \\
    \leq\|\boldsymbol{Q}^{k}- \hat{\boldsymbol{Q}}\|_{L^2(\omega)}\,\|\boldsymbol{y}_{,\alpha}^k \otimes\boldsymbol{\varphi}\|_{L^2(\omega)} + \displaystyle{\int_\omega}  \big( \boldsymbol{y}_{,\alpha}^k - \hat{\boldsymbol{y}}_{,\alpha}\big) \cdot\hat{\boldsymbol{Q}}\boldsymbol{\varphi} \,\mathrm{d}\omega.
\end{array}
\end{equation*}
Taking into account \eqref{34}, \eqref{37}$_2$ and the boundedness of $\{\boldsymbol{y}^k\}$ in $\boldsymbol{H}^1(\omega, \mathbb{R}^3)$, we deduce that $\|\boldsymbol{Q}^{k}- \hat{\boldsymbol{Q}}\|_{L^2(\omega)}\rightarrow 0$, and $ \|\boldsymbol{y}_{,\alpha}^k \otimes\boldsymbol{\varphi}\|_{L^2(\omega)} $ is bounded, and $\boldsymbol{y}_{,\alpha}^k \rightharpoonup \hat{\boldsymbol{y}}_{,\alpha}$ in $\boldsymbol{L}^2(\omega, \mathbb{R}^3)$. Then, from the last inequality we obtain
\begin{equation}\label{38+1}
    \displaystyle{\int_\omega}\big(\boldsymbol{Q}^{k,T} \boldsymbol{y}_{,\alpha}^k \big)\cdot\boldsymbol{\varphi} \,\mathrm{d}\omega \longrightarrow
    \displaystyle{\int_\omega}\big( \hat{\boldsymbol{Q}}^T \hat{\boldsymbol{y}}_{,\alpha}\big)\cdot\boldsymbol{\varphi} \,\mathrm{d}\omega,\quad\forall\, \boldsymbol{\varphi}\in  \boldsymbol{C}_0^\infty(\omega,\mathbb{R}^3).
\end{equation}
But from \eqref{36} we see that the sequence $\big\{\boldsymbol{Q}^{k,T} \boldsymbol{y}_{,\alpha}^k \big\}_{k=1}^\infty$ admits a weak limit $\hat{\boldsymbol{\ell}}$ in $\boldsymbol{L}^2(\omega,\mathbb{R}^3)$.
By virtue of \eqref{38+1} we infer that this weak limit $\hat{\boldsymbol{\ell}}$ must coincide with $\hat{\boldsymbol{Q}}^T \hat{\boldsymbol{y}}_{,\alpha}\,$. Thus, we have proved the convergence
\begin{equation*}
    \boldsymbol{Q}^{k,T} \boldsymbol{y}_{,\alpha}^k   \rightharpoonup
     \hat{\boldsymbol{Q}}^T \hat{\boldsymbol{y}}_{,\alpha}\quad\mathrm{in}\quad \boldsymbol{L}^2(\omega,\mathbb{R}^3), \,\,\alpha=1,2,
\end{equation*}
or, equivalently,
\begin{equation}\label{39}
    \boldsymbol{Q}^{k,T} \big[(\boldsymbol{y}^k_{,\alpha}\!-\boldsymbol{Q}^k\boldsymbol{e}_\alpha)\otimes \boldsymbol{e}_\alpha\big]\rightharpoonup
     \hat{\boldsymbol{Q}}^{T}\big[(\hat{\boldsymbol{y}}_{,\alpha}\!- \hat{\boldsymbol{Q}}\boldsymbol{e}_\alpha)\otimes \boldsymbol{e}_\alpha\big]
       \quad\mathrm{in}\quad \boldsymbol{L}^2(\omega,\mathbb{R}^3) ,
\end{equation}
which means (in view of \eqref{36}) that the relation \eqref{40}$_1$ holds.

Next, in order to prove the relation \eqref{40}$_2$ we proceed analogously: for any test function $\boldsymbol{\phi}\in \boldsymbol{C}_0^\infty(\omega,\mathbb{R}^{3\times3})$, we have
\begin{equation*}
\begin{array}{l}
    \displaystyle{\int_\omega}\big\langle \boldsymbol{Q}^{k} \boldsymbol{Q}_{,\alpha}^{k,T} - \hat{\boldsymbol{Q}} \hat{\boldsymbol{Q}}^T_{,\alpha}  \,,\, \boldsymbol{\phi} \big\rangle \,\mathrm{d}\omega=
    \displaystyle{\int_\omega}\big\langle \big( \boldsymbol{Q}^{k}- \hat{\boldsymbol{Q}}\big) \boldsymbol{Q}_{,\alpha}^{k,T} +  \hat{\boldsymbol{Q}} \big( \boldsymbol{Q}_{,\alpha}^{k,T} - \hat{\boldsymbol{Q}}_{,\alpha}^T\big) \,,\, \boldsymbol{\phi} \,\big\rangle \,\mathrm{d}\omega \vspace{2pt}\\
    =\displaystyle{\int_\omega} \big\langle \boldsymbol{Q}^{k}- \hat{\boldsymbol{Q}}\,,\, \boldsymbol{\phi} \boldsymbol{Q}_{,\alpha}^k \, \rangle \,\mathrm{d}\omega +
    \displaystyle{\int_\omega}
     \big\langle \boldsymbol{Q}_{,\alpha}^{k,T} - \hat{\boldsymbol{Q}}_{,\alpha}^T  \,,\, \hat{\boldsymbol{Q}}^T \boldsymbol{\phi} \,\big\rangle \,\mathrm{d}\omega  \\
    \leq\|\boldsymbol{Q}^{k}- \hat{\boldsymbol{Q}}\|_{L^2(\omega)}\,\,\|\boldsymbol{\phi} \boldsymbol{Q}_{,\alpha}^k \, \|_{L^2(\omega)} +\displaystyle{\int_\omega}
     \big\langle \boldsymbol{Q}_{,\alpha}^{k} - \hat{\boldsymbol{Q}}_{,\alpha}  \,,\, \boldsymbol{\phi}^T\hat{\boldsymbol{Q}} \,\big\rangle \,\mathrm{d}\omega.
\end{array}
\end{equation*}
Since the sequence $\big\{\boldsymbol{Q}_{,\alpha}^k\big\}_{k=1}^\infty$ is bounded in $\boldsymbol{L}^2(\omega,\mathbb{R}^{3\times3})$, in view of \eqref{37} we derive that the right-hand side of the last inequality tends to zero for $k\rightarrow\infty$. Thus, we obtain
\begin{equation}\label{38+3}
    \displaystyle{\int_\omega}\big\langle \boldsymbol{Q}^{k} \boldsymbol{Q}_{,\alpha}^{k,T}\,,\, \boldsymbol{\phi} \big\rangle \,\mathrm{d}\omega \longrightarrow
    \displaystyle{\int_\omega}\big\langle  \hat{\boldsymbol{Q}} \hat{\boldsymbol{Q}}^T_{,\alpha}\,,\, \boldsymbol{\phi} \big\rangle \,\mathrm{d}\omega,\quad\forall \,\boldsymbol{\phi}\in  \boldsymbol{C}_0^\infty(\omega,\mathbb{R}^{3\times3}),\,\,\,\alpha=1,2.
\end{equation}
By virtue of the relations \eqref{5} and \eqref{38}$_2$ we can write
\begin{equation*}
\begin{array}{l}
    \|\boldsymbol{K}^k\|^2_{L^2(\omega)}= \|\mathrm{axl}(\boldsymbol{Q}^k_{,\alpha}  \boldsymbol{Q}^{k,T} )\otimes \boldsymbol{e}_\alpha\|^2_{L^2(\omega)}= \displaystyle{\sum_{\alpha=1}^2} \,\|\,\mathrm{axl}(\boldsymbol{Q}^k_{,\alpha}  \boldsymbol{Q}^{k,T} ) \|^2_{L^2(\omega)}\\
    \qquad\qquad\quad \!\!= \dfrac{1}{2}\, \displaystyle{\sum_{\alpha=1}^2} \,\,\| \boldsymbol{Q}^k  \boldsymbol{Q}_{,\alpha}^{k,T} \|^2_{L^2(\omega)}\,.
\end{array}
\end{equation*}
Since   $\big\{\boldsymbol{K}^k\big\}_{k=1}^\infty$ is bounded in $\boldsymbol{L}^2(\omega,\mathbb{R}^{3\times3})$, we deduce from the above relations that the sequences $\big\{ \boldsymbol{Q}^k  \boldsymbol{Q}_{,\alpha}^{k,T}\big\}_{k=1}^\infty\,$ are also bounded ($\alpha=1,2$) and, hence, they admit weak limits $\,\hat{\boldsymbol{\ell}}_\alpha$ in $\boldsymbol{L}^2(\omega,\mathbb{R}^{3\times3})$. Now, in view of \eqref{38+3}    we obtain that $\,\hat{\boldsymbol{\ell}}_\alpha= \hat{\boldsymbol{Q}} \hat{\boldsymbol{Q}}^T_{,\alpha}\,$, i.e.
\begin{equation*}
    \boldsymbol{Q}^k  \boldsymbol{Q}_{,\alpha}^{k,T}\,\rightharpoonup\, \hat{\boldsymbol{Q}} \hat{\boldsymbol{Q}}^T_{,\alpha}\,\quad \mathrm{in}\quad \boldsymbol{L}^2(\omega,\mathbb{R}^{3\times3}),
\end{equation*}
or
\begin{equation}\label{38+4}
    \mathrm{axl}(\boldsymbol{Q}^k_{,\alpha}  \boldsymbol{Q}^{k,T} )\otimes \boldsymbol{e}_\alpha \,\,\rightharpoonup\,\, \mathrm{axl}(\hat{ \boldsymbol{Q}}_{,\alpha}\hat{\boldsymbol{Q}}^T)\otimes\boldsymbol{e}_\alpha \,\quad \mathrm{in} \quad \boldsymbol{L}^2(\omega,\mathbb{R}^{3\times3}).
\end{equation}
Taking \eqref{37}$_2\,$, \eqref{38} and \eqref{38+4} into account, and repeating the above argument (with arbitrary test functions $\boldsymbol{\phi}\in \boldsymbol{C}_0^\infty(\omega ,\mathbb{R}^{3\times3})$, similar to \eqref{38+3}) it follows that
\begin{equation*}
    \boldsymbol{Q}^{k,T} \big[\mathrm{axl}(\boldsymbol{Q}^k_{,\alpha}  \boldsymbol{Q}^{k,T} )\otimes \boldsymbol{e}_\alpha\big]\,\, \rightharpoonup \,\,  \hat{\boldsymbol{Q}}^T\big[\,\mathrm{axl}(\hat{ \boldsymbol{Q}}_{,\alpha}\hat{\boldsymbol{Q}}^T)\otimes\boldsymbol{e}_\alpha\big] \,\quad \mathrm{in}\quad \boldsymbol{L}^2(\omega,\mathbb{R}^{3\times3}),
\end{equation*}
which means that the relation \eqref{40}$_2$ holds.

In the next step of the proof, we show the convexity of the strain energy density function. By virtue of the conditions on the constitutive coefficients \eqref{23}, the Hessian matrix of the quadratic form $W_{\mathrm{mb}}(\boldsymbol{E})$ in \eqref{26} is positive definite. Similar arguments hold for $W_{\mathrm{bend}}(\boldsymbol{K})$, and altogether  we obtain
\begin{equation}\label{44}
    W(\boldsymbol{E},\boldsymbol{K}) \,\,\,\,\text{is convex in}\,\,\,\, (\boldsymbol{E},\boldsymbol{K}).
\end{equation}
By  \eqref{36},\eqref{38} and \eqref{40} we deduce $\boldsymbol{E}^k \rightharpoonup \hat{\boldsymbol{E}} $ and  $ \boldsymbol{K}^k \rightharpoonup \hat{\boldsymbol{K}} $ in $\boldsymbol{L}^2(\omega)$, and from \eqref{44} we find that
\begin{equation}\label{45}
    \int_\omega W(\hat{\boldsymbol{E}},\hat{\boldsymbol{K}})\,\mathrm{d}\omega \leq \liminf_{k\to\infty}  \int_\omega W( {\boldsymbol{E}^k}, {\boldsymbol{K}^k})\,\mathrm{d}\omega.
\end{equation}
If we denote by $\boldsymbol{u}^k=\boldsymbol{y}^k-\boldsymbol{x}$ , $\hat{\boldsymbol{u}}=\hat{\boldsymbol{y}}-\boldsymbol{x}$, then from \eqref{20}--\eqref{22},\eqref{35},\eqref{37}$_2$ and the continuity of $\Pi_\omega\,$, $\Pi_{\partial\omega_f}$ it follows
\begin{equation}\label{46}
    \lim_{k\to\infty} \Lambda(\boldsymbol{u}^k, \boldsymbol{Q}^k)=\Lambda(\hat{\boldsymbol{u}}, \hat{\boldsymbol{Q}}).
\end{equation}
Since the pairs $(\boldsymbol{y}^k, \boldsymbol{Q}^k)$ satisfy the boundary conditions on $\partial\omega_d\,$, we deduce in view of the
convergence relations \eqref{34}, \eqref{35}, \eqref{37} and the compact embedding in the sense of traces, that $\hat{\boldsymbol{y}}=\boldsymbol{y}^*,\,\, \hat{\boldsymbol{Q}}=\boldsymbol{Q}^*$ on  $\partial\omega_d\,$. Hence, we have $(\hat{\boldsymbol{y}}, \hat{\boldsymbol{Q}})\in\mathcal{A}$.

Finally, from \eqref{32},\eqref{45} and \eqref{46} we obtain $I(\hat{\boldsymbol{y}},\hat{\boldsymbol{Q}})\leq \inf \{ I(\boldsymbol{y},\boldsymbol{Q})\,|\,(\boldsymbol{y},\boldsymbol{Q})\in \mathcal{A}\}$, which means that $(\hat{\boldsymbol{y}},\hat{\boldsymbol{Q}})$ is a minimizer of the functional $I$ over $\mathcal{A}$. The proof is complete.
\hfill\end{proof}

\begin{remark}\label{rem3}
We observe that $\hat{\boldsymbol{Q}}\in\boldsymbol{L}^\infty(\omega,SO(3))$ since $\|\hat{\boldsymbol{Q}}\|^2=3$, but the rotation $\hat{\boldsymbol{Q}}$ may fail to be continuous, according to the limit case of the Sobolev embedding. Also, the solution deformation  $\,\hat{\boldsymbol{y}}\in \boldsymbol{H}^1(\omega,\mathbb{R}^{3})$ may fail to be continuous. This is indeed an advantage of our formulation, since the results are applicable to a large class of  shells (such as e.g., shells with singularity lines), and the minimizing solutions are not restricted by too  strong regularity conditions.
\end{remark}
\begin{remark}\label{rem4}
We notice that the conditions on the constitutive coefficients \eqref{23} are satisfied for the  particular model of isotropic plates presented in Remark \ref{rem1}. Indeed, taking into account the identification \eqref{17}, we find that the inequalities \eqref{23} reduce to
$$\mu>0,\quad 2\mu + 3\lambda>0.$$
These conditions are satisfied in view of the positive definiteness of the three-dimensional quadratic elastic strain energy density for isotropic materials. Thus, the existence result given by Theorem \ref{th1} applies to the particular   plate model presented in  \cite{Pietraszkiewicz-book04,Pietraszkiewicz10}.
\end{remark}

\section{Generalization of existence result}\label{sect4}

In this section we present some variants and generalizations of Theorem \ref{th1}.

We observe that the boundary conditions imposed on the rotation $\boldsymbol{Q}$ can be relaxed or even omitted in the definition of the admissible set \eqref{18}. For a discussion of some possible alternative boundary conditions for the rotation field $\boldsymbol{Q}$ on $\partial\omega_d$ we refer to the works \cite{Neff_plate04_cmt,Neff_plate07_m3as}.
In this line of thought, we present next the existence result corresponding to a larger admissible set.
\begin{theorem}\label{th2}
Consider the minimization problem \eqref{19}, over the admissible set
\begin{equation}\label{47}
    \widetilde{\mathcal{A}} =\big\{(\boldsymbol{y},\boldsymbol{Q})\in\boldsymbol{H}^1(\omega,\mathbb{R}^3)\times\boldsymbol{H}^1(\omega, SO(3))\,\,\, \big|\,\, \,\, \boldsymbol{y}_{\big| \partial\omega_d}=\boldsymbol{y}^*  \big\}.
\end{equation}
If the external loads $\boldsymbol{f}$, $\boldsymbol{n}^*$ and the boundary data $\boldsymbol{y}^*$ satisfy the conditions \eqref{22}$_{1-3}$ and the constitutive coefficients $\alpha_k\,,\,\beta_k$  verify the inequalities \eqref{23}, then the minimization problem \eqref{19}, \eqref{47} admits at least one minimizing solution pair $(\hat{\boldsymbol{y}},\hat{\boldsymbol{Q}})\in\widetilde{\mathcal{A}}$.
\end{theorem}

\begin{proof}The proof can be achieved in a similar manner as the proof of Theorem \ref{th1}, where the boundary condition $\boldsymbol{Q}=\boldsymbol{Q}^*$ on $\partial\omega_d$ has not played an important role.
{}\hfill\end{proof}
\bigskip

The Theorem \ref{th1} is concerned with isotropic plates for which the strain energy density $W( {\boldsymbol{E}}, {\boldsymbol{K}})$ is  given by relations \eqref{14}.
We can generalize this existence result to the  case of anisotropic non-linear plates, provided the function $W$ satisfies the conditions of convexity and coercivity:
\begin{theorem}\label{th3}
\textrm{\textbf{(Anisotropic plates)}} Consider the minimization problem \eqref{18}, \eqref{19} associated to the deformation of  anisotropic  plates, and assume that the external loads and boundary data satisfy the conditions \eqref{22}. Assume that the strain energy density $W( {\boldsymbol{E}}, {\boldsymbol{K}})$ is an arbitrary quadratic convex function in $( {\boldsymbol{E}}, {\boldsymbol{K}})$, and moreover $W$ is coercive, in the sense that
\begin{equation}\label{48}
    W( {\boldsymbol{E}}, {\boldsymbol{K}})\geq k\big(    \|\boldsymbol{E}\|^2 +  \|\boldsymbol{K}\|^2\big),
        \quad\forall\,\boldsymbol{E}=E_{i\alpha}\boldsymbol{e}_i\otimes\boldsymbol{e}_\alpha, \, \boldsymbol{K}=K_{i\alpha}\boldsymbol{e}_i\otimes\boldsymbol{e}_\alpha,\quad \!\!
    E_{i\alpha}, K_{i\alpha}\in\mathbb{R},
\end{equation}
for some constant $k> 0$. Then, the minimization problem \eqref{18}, \eqref{19} admits at least one minimizing solution pair $(\hat{\boldsymbol{y}},\hat{\boldsymbol{Q}})\in \mathcal{A}$.
\end{theorem}
\begin{proof}
We follow the same steps as in the proof of Theorem \ref{th1}. In view of \eqref{48}, we show first that the estimate \eqref{31} holds also in our case. Then, there exists a minimizing sequence $\big\{\boldsymbol{y}^k,\boldsymbol{Q}^k\big\}^\infty_{k=1}\subset\mathcal{A}$, and we can prove similarly that it verifies the relations \eqref{32}-\eqref{40}. By virtue of our hypothesis, $W$ is  convex  in $( {\boldsymbol{E}}, {\boldsymbol{K}})$, so that the properties \eqref{44}-\eqref{46} are satisfied, and we can reach the conclusion of the theorem.
\hfill\end{proof}
\medskip

The Theorem \ref{th3} remains valid also for the minimization  problem written over the larger admissible set $\widetilde{\mathcal{A}}$ given by \eqref{47}, instead of the admissible set \eqref{18}.
\begin{remark}\label{rem5}
The 6-parametric theory of shells can be used to model also composite  thin elastic structures.  In this case, the internal energy density has a more complicated structure, and exhibits multiplicative coupling of the strain tensor ${\boldsymbol{E}}$ with the bending tensor ${\boldsymbol{K}}$, see e.g. \cite{Chroscielewski11,Kreja07}. Nevertheless, the Theorem \ref{th3} can be applied to deduce the existence of minimizers for layered composite plates, under appropriate conditions on the material/geometrical parameters \cite{Birsan-Neff-AnnRom12}.
\end{remark}

\section{Comparison with a Cosserat model for plates}\label{sect5}

In this section we present a comparison with the Cosserat model for plates (shells) proposed and investigated  by the second author in \cite{Neff_plate04_cmt,Neff_plate07_m3as}. This model is obtained by a consistent formal dimensional reduction of a finite-strain three-dimensional Cosserat (micropolar) model to the two-dimensional situation of thin plates.

Apart from the    differences in notations, there are essential similarities between this Cosserat plate model and the 6-parametric plate theory presented in Section \ref{sect1}. Firstly, in both approaches the primary independent kinematical variables are the deformation field $\boldsymbol{y}\in \mathbb{R}^3$ and the rotation tensor field $\boldsymbol{Q}\in SO(3)$, which are denoted in \cite{Neff_plate04_cmt,Neff_plate07_m3as} by $\boldsymbol{m}$ and $\overline{\boldsymbol{R}}$. Thus, we have the correspondence
$$\boldsymbol{y}=\boldsymbol{m},\qquad \boldsymbol{Q}=\overline{\boldsymbol{R}}.$$
The rotation field $\overline{\boldsymbol{R}}$ in this derivation approach is inherited from the parent Cosserat bulk model, which already includes a triad of rigidly rotating directors (the ``triedre mobile'', see \cite{Neff_Paris_Maugin09}).

Moreover, the measures of strain are essentially the same in the two approaches. Indeed, the so-called \emph{stretch tensor}  $\overline{\boldsymbol{U}}$ and the third order \emph{curvature tensor} $\boldsymbol{\mathfrak{K}}_s$ are introduced in \cite{Neff_plate04_cmt} Sect. 4, through the relations
\begin{equation*}
\begin{array}{l}
\overline{\boldsymbol{U}}\,= \overline{\boldsymbol{R}}^T(\boldsymbol{m}_{,\alpha}\otimes\boldsymbol{e}_\alpha+ \overline{\boldsymbol{R}}_3\otimes \boldsymbol{e}_3),
\\
\boldsymbol{\mathfrak{K}}_s= \big[\,\overline{\boldsymbol{R}}^T(\overline{\boldsymbol{R}}_{i,\alpha}\otimes \boldsymbol{e}_\alpha)\big]\otimes \boldsymbol{e}_i \,, \quad\mathrm{with}\quad \overline{\boldsymbol{R}}_i=\overline{\boldsymbol{R}}\, \boldsymbol{e}_i\,.
    \end{array}
\end{equation*}
If we compare these tensors with the strain measures ${\boldsymbol{E}}$ and  ${\boldsymbol{K}}$ defined in \eqref{8}-\eqref{11}, then we find that
$${\boldsymbol{E}}=\overline{\boldsymbol{U}}-\id\, ,$$
where $\id=\boldsymbol{e}_i\otimes\boldsymbol{e}_i$ is the identity tensor.
It can be shown that the tensor $\boldsymbol{\mathfrak{K}}_s$ has only 6 distinct components which coincide with the 6 non-zero components of the bending tensor ${\boldsymbol{K}}$ given by \eqref{11}.

Let us compare now the expressions for the strain energy density. The strain energy density is assumed in the Cosserat plate model in the form
\begin{equation}\label{49}
\begin{array}{l}
    W(\overline{\boldsymbol{U}}, \boldsymbol{\mathfrak{K}}_s)= h\Big(\mu\,\|\mathrm{sym}(\overline{\boldsymbol{U}}-\id)\|^2 + \mu_c\|\mathrm{skew}(\overline{\boldsymbol{U}}-\id)\|^2+ \dfrac{\lambda\mu}{\lambda+2\mu}\, \big(\mathrm{tr}(\overline{\boldsymbol{U}}-\id)\big)^2\Big)   \\
    \quad + \dfrac{h^3}{12}\,\Big(\,\mu\|\mathrm{sym}\boldsymbol{\mathfrak{K}}_s^3\|^2 + \mu_c\|\mathrm{skew}\boldsymbol{\mathfrak{K}}_s^3\|^2+ \dfrac{\lambda\mu}{\lambda+2\mu}\, \big(\mathrm{tr}(\boldsymbol{\mathfrak{K}}_s^3)\big)^2\Big)\\
    \quad +\dfrac{ hL_c^{1+p}\mu}{12}\,\big(1\!+\!a_4L_c^q\|\boldsymbol{\mathfrak{K}}_s\|^q\big)\displaystyle{\sum_{i=1}^3}\Big( a_5
    \|\mathrm{sym}\boldsymbol{\mathfrak{K}}_s^i\|^2 + a_6\|\mathrm{skew}\boldsymbol{\mathfrak{K}}_s^i\|^2+ a_7 \big(\mathrm{tr}(\boldsymbol{\mathfrak{K}}_s^i)\big)^2\Big)^{\frac{1+p}{2}},
\end{array}
\end{equation}
where $\boldsymbol{\mathfrak{K}}_s^i= \boldsymbol{\mathfrak{K}}_s\boldsymbol{e}_i\,$, and $\mu_c\geq 0$ is the Cosserat couple modulus, $L_c$ is an internal length, and $a_4,...,a_7$ are constitutive coefficients introduced in \cite{Neff_plate04_cmt}. The exponents $p$ and $q$  are such that $p\geq 1$, $q\geq 0$. The existence of minimizers is proved for the case of the strain energy \eqref{49} under the simple conditions
\begin{equation}\label{50}
    \mu_c>0,\quad \mu>0,\quad \lambda>0,\quad a_5>0,\quad  a_6>0,\quad a_7\geq 0.
\end{equation}
If we choose the exponent $p=1$ and the parameter $a_4=0$ in \eqref{49}, then the expressions of the strain energy density in the two approaches coincide, in terms of the independent kinematical variables $(\boldsymbol{y},\boldsymbol{Q})=(\boldsymbol{m}, \overline{\boldsymbol{R}})$. To realize the coincidence of the two strain energy functions \eqref{14} and \eqref{49} in this case, we need to identify the set of constitutive coefficients $(\alpha_1,...,\alpha_4,\beta_1,...,\beta_4)$ from \eqref{14} with the set of  parameters $(\lambda,\mu,\mu_c,L_c, a_5, a_6, a_7)$ in the following way
\begin{equation}\label{51}
\begin{array}{l}
    \alpha_1=h\,\dfrac{2\lambda\mu}{\lambda+2\mu}\,,\quad \alpha_2=h(\mu-\mu_c),\quad \alpha_3=h(\mu+\mu_c),\quad \alpha_4=\kappa h (\mu+\mu_c),\vspace{3pt}\\
    \beta_1=-\dfrac{h}{12}\,\big[h^2(\mu-\mu_c)+\mu L_c^2\, (a_5-a_6)\big],\quad\, \, \beta_2=-\dfrac{h\mu}{6}\,\big(h^2\dfrac{\lambda}{\lambda\!+\!2\mu} + L_c^2\, a_7\big),\vspace{3pt}\\
    \beta_3=\dfrac{h\mu}{6}\,\big[h^2\dfrac{2(\lambda\!+\!\mu)}{\lambda\!+\!2\mu} + L_c^2\, (\frac{3}{2}\,a_5+\frac{1}{2}\,a_6+ a_7)\big], \quad
    \beta_4=\dfrac{h\mu}{6}\,  L_c^2 \,(\frac{3}{2}\,a_5+\frac{1}{2}\,a_6+ a_7),
\end{array}
\end{equation}
where $\kappa$ is the formal shear correction factor. We observe that
\begin{equation}\label{52}
    \alpha_3-\alpha_2=2h\,\mu_c\,\,,
\end{equation}
so that the condition $\mu_c>0$ considered in \eqref{50}$_1$ corresponds to $\alpha_3-\alpha_2>0$, assumed in the conditions \eqref{23} which ensure the positive definiteness of $W$. The interesting degenerate case $\mu_c=0$ is investigated in details in \cite{Neff_plate07_m3as}. In view of \eqref{52}, this case corresponds to $\alpha_3-\alpha_2=0$, when the energy function $W$ is only positive semi-definite, and the proof of the existence results is more delicate \cite{Neff_plate07_m3as}.

There are also some differences between the two approaches, such as for instance the form of the boundary conditions and the expression of the external loading potential $\Lambda( {\boldsymbol{u}},  {\boldsymbol{Q}})$. The proof of Theorem \ref{th1} follows the same steps as the proof of Theorem 4.1 in \cite{Neff_plate04_cmt}. However, the conditions on the constitutive coefficients \eqref{50} imposed in  \cite{Neff_plate04_cmt} are more restrictive than the conditions \eqref{23} in Theorem \ref{th1}.
For example, in the case of the isotropic plate model described in Remark \ref{rem1}, the constitutive coefficients \eqref{17} satisfy the existence conditions \eqref{23} stated by Theorem \ref{th1}, but they do not verify the restrictions \eqref{50} (in virtue of the identification \eqref{51}). This illustrates the fact that the  conditions for the constitutive coefficients established in the present work are less restrictive.

Finally, we remark that the form of the strain energy density $W$ given by \eqref{49} is more general than \eqref{14}. Indeed, the expression \eqref{14} of $W$ can be obtained from \eqref{49} if we choose the parameters $p=1$ and $a_4=0$. Nevertheless, the existence of minimizers is proved in  \cite{Neff_plate04_cmt} for any exponents $p$, $q$ and coefficient $a_4$ satisfying the conditions: $p\geq 1$, $q\geq 0$, and $a_4\geq 0$.

In a future contribution we will extend our results to the general case of 6-parameter shells.

\bigskip\bigskip
\small{\textbf{Acknowledgements.}
The first author (M.B.) is supported by the german state  grant: ``Programm des Bundes und der L\"ander f\"ur bessere Studienbedingungen und mehr Qualit\"at in der Lehre''. Useful discussions with Professor V.A. Eremeyev are gratefully acknowledged.


\bibliographystyle{plain} 
{\footnotesize
\bibliography{literatur_Birsan}
}


\end{document}